\documentclass{amsart}

\usepackage{amsmath,amsfonts,amsthm,amssymb,graphics,graphicx, verbatim,color}

\newtheorem{thm}{Theorem}[section]

\newtheorem{lem}[thm]{Lemma}

\newtheorem{prop}[thm]{Proposition}
\theoremstyle{remark} \newtheorem{remark}[thm]{Remark}
\newcommand\edit[1]{{\color{black}#1}}

\newcommand{\norm}[1]{\left|\!\left|{#1}\right|\!\right|}
\newcommand{\R}{\ensuremath{\mathbb{R}}}

\newcommand\Id{\operatorname{Id}}
\renewcommand\Re{\operatorname{Re}}

\newcommand\chiev{\chi^{\mathrm{ev}}}
\newcommand{\indic}{\operatorname{1\negthinspace l}}

\title[Eigenfunction estimates on manifolds of nonpositive curvature]{Improvement of eigenfunction estimates on manifolds of nonpositive curvature}

\author{Andrew Hassell and Melissa Tacy}

\email{Andrew.Hassell@anu.edu.au}
\email{mtacy@math.northwestern.edu}

\address{Mathematical Sciences Institute, Australian National University, Canberra 0200 ACT Australia}
\address{Department of Mathematics, Northwestern University, Evanston 60208 IL, USA}


\thanks{A. H. was supported by Australian Research Council Future Fellowship FT0990895 and Discovery Grants DP1095448 and DP120102019.}

\begin{document}
\begin{abstract}
Let $(M,g)$ be a compact, boundaryless manifold of dimension $n$ with the property that either (i) $n=2$ and $(M,g)$ has no conjugate points, or (ii) the sectional curvatures of $(M,g)$ are nonpositive. Let $\Delta$ be the positive Laplacian on $M$ determined by $g$. 
We study the $L^{2}\to{}L^{p}$ mapping properties of a spectral cluster of $\sqrt{\Delta}$  of width $1/\log\lambda$. Under the geometric assumptions above, \cite{berard77} B\'{e}rard obtained a logarithmic improvement for the remainder term of the eigenvalue counting function which directly leads to a $(\log\lambda)^{1/2}$ improvement for H\"ormander's estimate on the $L^{\infty}$ norms of eigenfunctions. In this paper we extend this improvement to the $L^p$ estimates for all $p>\frac{2(n+1)}{n-1}$.
\end{abstract}
\maketitle

\section{Introduction}
We study the growth of the $L^p$ norms of high eigenvalue eigenfunctions of the Laplace-Beltrami operator on compact manifolds. That is, we seek estimates of the form
\begin{equation}
\norm{u}_{L^{p}(M)}\lesssim{}f(\lambda,n,p)\norm{u}_{L^{2}}\label{estimatef}
\end{equation}
where $u$ is an eigenfunction, $\Delta{}u=\lambda^{2}u$, of the Laplace-Beltrami operator on a $n$-dimensional, compact, boundaryless Riemannian manifold $(M,g)$
(we adopt the sign convention that the Laplacian $\Delta$ is a positive operator). As is well known, $L^{2}\to{}L^{\infty}$ estimates for eigenfunctions follow directly from eigenvalue multiplicity bounds, and  H\"ormander's counting function remainder estimates \cite{hormander68} yield a bound 
$$\norm{u}_{L^{\infty}}\lesssim\lambda^{\frac{n-1}{2}}\norm{u}_{L^{2}}.$$
For general manifolds without any additional geometric assumptions Sogge \cite{sogge88} obtained $L^{p}$ estimates of the form
\begin{equation}\begin{gathered}
f(\lambda,n,p)=\lambda^{\delta(n,p)} \\
\delta(n,p)=\begin{cases}
\frac{n-1}{2}-\frac{n}{p} & \frac{2(n+1)}{n-1}\leq{}p\leq\infty\\
\frac{n-1}{4}-\frac{n-1}{2p}  & 2\leq{}p\leq\frac{2(n+1)}{n-1}\end{cases}.
\end{gathered}\label{deltanp}\end{equation}
  These estimates are in fact for spectral clusters of window width one. While they are sharp for such clusters the only known sharp eigenfunction examples are in cases of high multiplicity of the spectrum. For the $L^{2}\to{}L^{\infty}$ estimates more is known. Sogge and Zelditch \cite{sogge02} and Sogge, Zelditch and Toth \cite{sogge11} investigated the conditions on $M$ required for maximal $L^{\infty}$ growth. They determined that to achieve the sharp $L^{\infty}$ bound given by \eqref{deltanp} it is necessary that at some point $x\in{}M$ both that the set of  directions in $S^{\star}M$  that loop back to $x$ has positive measure and that the first return map be recurrent. As $L^{2}\to{}L^{\infty}$ estimates follow directly from multiplicity estimates  an improvement in the remainder term of the counting function automatically implies an improvement in $L^{\infty}$ estimates.  It is expected that systems whose classical dynamics exhibit chaotic behaviour will have lowered multiplicity and therefore lowered $L^{\infty}$ growth. In 1977 B\'{e}rard \cite{berard77} obtained a $\log\lambda$ improvement for  the remainder term in the counting function (and therefore an improved multiplicity bound) for manifolds with nonpositive section curvature (in two dimensions, the condition of no conjugate points suffices). His result directly implies an improvement to the $L^{\infty}$ bounds of
  $$\norm{u}_{L^{\infty}}\lesssim\frac{\lambda^{\frac{n-1}{2}}}{(\log\lambda)^{1/2}}\norm{u}_{L^{2}}.$$
  By interpolation with Sogge's $p=\frac{2(n+1)}{n-1}$ estimate we can therefore easily obtain
  $$\norm{u}_{L^{p}}\lesssim\frac{\lambda^{\frac{n-1}{2}-\frac{n}{p}}}{(\log\lambda)^{\frac{1}{2}-\frac{n+1}{p(n-1)}}}\norm{u}_{L^{2}}\quad{}p\geq{}\frac{2(n+1)}{n-1}.$$
  In this paper we show that under the same assumptions as in B\'{e}rard \cite{berard77}, the improvement by a factor of $(\log \lambda)^{1/2}$ persists for $p>\frac{2(n+1)}{n-1}$. That is, our main result is 
\begin{thm}\label{maintheorem}
Let $u$ be \edit{an} eigenfunction of the Laplacian $-\Delta{}u=\lambda^{2}u$ on a compact boundaryless Riemannian manifold $(M,g)$ of dimension $n$, such that either (i) $n=2$ and $(M,g)$ has no conjugate points or (ii) the sectional curvatures of $(M,g)$ are nonpositive. Then we have the estimate
$$\norm{u}_{L^{p}}\leq{}C_{p}\frac{\lambda^{\frac{n-1}{2}-\frac{n}{p}}}{(\log\lambda)^{1/2}}$$
in the range
\begin{equation}
p>\frac{2(n+1)}{n-1}. 
\label{prange}\end{equation}
\end{thm}

We remark that the constant $C_p$ we obtain blows up as 
%  $$\norm{u}_{L^{p}}\leq{}C_{p}\frac{\lambda^{\frac{n-1}{2}-\frac{n}{p}}}{(\log\lambda)^{1/2}}\norm{u}_{L^{2}}\quad{}p>\frac{2(n+1)}{n-1}$$
%  with a constant $C_{p}$ that blows up as 
$p$ tends to the `kink point' $\frac{2(n+1)}{n-1}$. It remains unknown whether a logarithmic improvement, of any power, is valid at $p=\frac{2(n+1)}{n-1}$. For surfaces there are some similar results known below the kink point, that is $p<6$. Bourgain \cite{bourgain09a} and Sogge \cite{sogge11a} study the relationship between $\norm{u}_{L^{4}(M)}$ and $\norm{u}_{L^{2}(\gamma)}$ where $\gamma$ is a geodesic. Sogge shows that to improve the $L^{p}(M)$ estimates for $2<p<6$ it is necessary and sufficient to improve the $L^{2}$ restriction estimates. In related work Sogge and Zelditch \cite{sogge12} and Chen and Sogge \cite{chen12a} study improvements to geodesic restriction theorems for low $p$ in the case of nonpositive curvature. In the high $p$ range  Chen \cite{chen12} obtains logarithmic improvements to the restriction theorems of Burq, G\'{e}rard and Tvetkov \cite{burq07}. Also in the setting of nonpositive curvature Bourgain, Shao, Sogge and Yao \cite{bourgain12} obtain logarithmic improvements for $L^{p}$ norms of resolvent operators.

At this point we note the history of this paper. Theorem \ref{maintheorem} was originally presented by the second author at a conference at Dartmouth College in 2010 in the special case of constant negative curvature. Following the general techniques outlined in that presentation Chen \cite{chen12} in 2012 proved restriction estimates, in the general nonpositive curvature case, which are logarithmic improvements on estimates of Burq, G\'erard and Tzvetkov \cite{burq07}. This paper presents the original result in this more general setting.
  
  To obtain eigenfunctions estimates we study a smoothed spectral cluster $\chiev_{\lambda,w}(\sqrt{\Delta})$ defined in the following section. Here the spectral cluster is centred at $\lambda$ and has effective width $w$. This spectral cluster operator is based on  Sogge's construction \cite[Section 5.1]{sogge93}, with a slight variation as in \cite[Section 8.1]{GHS}.  Sogge's estimates correspond to  the case $w=1$. B\'{e}rard obtained his logarithmic improvements by shrinking the window width to $1/\log\lambda$. We will likewise aim to shrink window widths by a logarithmic factor. 
   B\'{e}rard's method relied on expressing the solution operator the the wave equation  on $M$ as a sum of shifted solutions operators on the universal cover of $M$, denoted $\widetilde{M}$. If $M$ has no conjugate points, $\widetilde{M}$ has an infinite injectivity radius. The $\log\lambda$ improvement is then obtained by estimating the contribution from each copy of the fundamental domain and estimating the number of copies that can be reached by finite speed propagation in $\log\lambda$ time.
   
%%%%%%%%%%%%%%%%%%%%%%%%%%%%%%%%%%%%%%%%%%
%%%%%%%%%%%%%%%%%%%%%%%%%%%%%%%%%%%%%%%%%%
   
\section{Spectral cluster operator}
Our spectral cluster operator $\chiev_{\lambda,w}(\sqrt{\Delta})$ is defined as follows. 
 Let $\chi$ be a Schwartz function such that $\chi(0)=1$ and $\hat{\chi}$ is supported in $[\epsilon,2\epsilon]$. Then that $\chi(0) > 0$, and by taking $\epsilon$ sufficiently small, we can arrange that $\Re \chi \geq c > 0$ on $[0,1]$. We then define, following \cite{GHS}, and for $0 < w \leq 1$, 
  $$\chiev_{\lambda,w}(\mu)=\chi\left(\frac{\mu-\lambda}{w}\right) + \chi\left(\frac{-\mu-\lambda}{w}\right).$$
  The first term on the right hand side is, for $w=1$, precisely Sogge's spectral cluster operator. 
  Notice that $\chiev_{\lambda,w}(\mu)$ is an even function of $\mu$ and for $\lambda$ large enough we have 
 $$
\Re \chiev_{\lambda, w} \geq \frac{c}{2} \text{ on } [\lambda, \lambda + w].
$$ 
Hence we can write 
$$
(\Re \chiev_{\lambda, w})^2 - \frac{c^2}{8} = F_{\lambda,w}
$$
for some function $F_{\lambda, w}$ that is nonnegative on the interval $[\lambda, \lambda + w]$.
 Let $u$ be an $L^2$-normalized eigenfunction with eigenvalue $\lambda$. Then the $L^p$ norm of $u$ is majorized by the $L^2  \to L^p$ operator norm of $\indic_{[\lambda, \lambda + w]}(\sqrt{\Delta})$, or equivalently the $L^{p'} \to L^2$ norm of the same operator. However, for $f \in L^{p'}$, we compute as in \cite[Section 8.1]{GHS}
\begin{equation*}\begin{gathered}
\frac{c^2}{8} \big\| \indic_{[\lambda, \lambda + w]}(\sqrt{\Delta}) f \big\|_{L^2}^2 = \Big\langle \indic_{[\lambda, \lambda + w]}(\sqrt{\Delta}) f , (\Re \chiev_{\lambda, w}(\sqrt{\Delta}))^2 -  F_{\lambda, w}(\sqrt{\Delta}) \Big) f \Big\rangle \\
%= \big\| \indic_{[\lambda, \lambda + 1]}(\sqrt{\Delta_N}) \Big( 
%(\Re \chiev(\sqrt{\Delta_N}))^2 -  F_\lambda(\sqrt{\Delta_N})
%\Big) f \big\|_{L^2}^2 \\
= \Big\langle \indic_{[\lambda, \lambda + w]}(\sqrt{\Delta})  
\Re \chiev_{\lambda, w}(\sqrt{\Delta}) f, 
%\indic_{[\lambda, \lambda + 1]}(\sqrt{\Delta_N}) 
\Re \chiev_{\lambda, w}(\sqrt{\Delta}) f \Big\rangle \\ - 
\Big\langle F_{\lambda, w}(\sqrt{\Delta}) \indic_{[\lambda, \lambda + \edit{w}]}(\sqrt{\Delta}) f, \indic_{[\lambda, \lambda + w]}(\sqrt{\Delta}) f \Big\rangle \\
\leq \big\| \Re \chiev_{\lambda, w}(\sqrt{\Delta}) f \big\|_{L^2}^2  
\leq \big\| \chiev_{\lambda, w}(\sqrt{\Delta}) f \big\|_{L^2}^2.
\end{gathered}\end{equation*}
So to obtain estimates  for Laplacian eigenfunctions it is enough to estimate the operator norm of the operator $\chiev_{\lambda, w}$  from $L^{p'}$ to $L^2$ for any $w$, $0 < w \leq 1$. 
%  Note that if $u$ is an eigenfunction of eigenvalue $\lambda^{2}$ then for any $w$, $\chi_{\lambda}(w)u=u$. Therefore eigenfunctions estimates follow from $L^{2}\to{}L^{p}$ mapping estimates for $\chi_{\lambda}(w)$. 
The parameter $w$ controls the effective size of the spectral window; we aim to make $w$ as small as possible. We have
 \begin{equation}\begin{gathered}
 \chi_{\lambda,w}(\sqrt{\Delta})u=\int{}\Big( e^{it\sqrt{\Delta}/w} + e^{-it\sqrt{\Delta}/w} \Big) e^{-it\lambda/w} \hat\chi(t)u \, dt \\
 = 2w \int{} \cos(t\sqrt{-\Delta}) e^{-it\lambda} \hat\chi(wt)u \, dt.
 \end{gathered}\end{equation}

%\begin{thm}\label{maintheorem}
%Let $u$ be a eigenfunction of the Laplacian $-\Delta{}u=\lambda^{2}u$ on a compact boundaryless Riemannian manifold $(M,g)$ without conjugate points  then
%$$\norm{u}_{L^{p}}\leq{}C_{p}\frac{\lambda^{\frac{n-1}{2}-\frac{n}{p}}}{(\log\lambda)^{1/2}}$$
%in the range
%\begin{equation}
%p>\frac{2(n+1)}{n-1}
%\label{prange}\end{equation}
%\end{thm}

\begin{remark}
The point of considering $\chiev_{\lambda, w}(\sqrt{\Delta})$ rather than $\chi_\lambda(\sqrt{\Delta}/w)$ is that the former operator can be expressed in terms of the cosine kernel, i.e. the kernel of $\cos t \sqrt{\Delta}$, instead of the half-wave operator $e^{it \sqrt{\Delta}}$. This allows us to exploit the finite propagation speed of the cosine kernel (see Proposition \ref{T2estimate}). 
\end{remark}

%\begin{remark}
%The constant $C_{p}$ blows up as $p$ approaches $\frac{2(n+1)}{n-1}$ and indeed we discover that as we approach this critical value of $p$ it becomes necessary to enlarge the window width of the spectral cluster. 
%
%\end{remark}

%%%%%%%%%%%%%%%%%%%%%%%%%%%%%%%%%%%%%%%%%%%
%%%%%%%%%%%%%%%%%%%%%%%%%%%%%%%%%%%%%%%%%%%

\section{Proof of Theorem \ref{maintheorem}}
We work with a smoothed spectral cluster  of window width $1/A$, we will allow the parameter $A$ to remain free at the moment and later set it as required.  With $\chi$ defined as above,  we need $L^{2}\to{}L^{p}$ mapping estimates for 
\begin{equation}
\chiev_{\lambda, 1/A}(\sqrt{\Delta}) = \chi\left(A(\sqrt{\Delta}-\lambda)\right) + \chi\left( A(-\sqrt{\Delta}-\lambda)\right).
\label{chidef}
\end{equation}
The proof of Theorem \ref{maintheorem} reduces to showing that for any $p$ in the range \eqref{prange} there exists an $\alpha_{p}$ and a $C_{p}$ such that for $A=\alpha_{p}\log\lambda$
$$\norm{\chiev_{\lambda, 1/A}(\sqrt{\Delta})u}_{L^{p}}\leq{}C_{p}\frac{\lambda^{\delta(n,p)}}{(\log\lambda)^{1/2}}\norm{u}_{L^{2}}.$$

 We  approach this problem via a $TT^{\star}$ method. That is we seek estimates of the form
$$\norm{\chiev_{\lambda, 1/A}(\sqrt{\Delta})\chiev_{\lambda, 1/A}(\sqrt{\Delta})^{\star}u}_{L^{p}}\leq{}C_{p}\frac{\lambda^{2\delta(n,p)}}{A}\norm{u}_{L^{p'}}$$
We have that
$$\chiev_{\lambda, 1/A}(\sqrt{\Delta}) u=2\int{}e^{-it\lambda{}A}\cos \big( t{}A\sqrt{\Delta} \big) \hat{\chi}(t)u \, dt$$
so
$$\chiev_{\lambda, 1/A}(\sqrt{\Delta})^{\star}u=2\int{}e^{is\lambda{}A}\cos \big(s{}A\sqrt{\Delta} \big)\hat{\chi}(s)uds$$
therefore it suffices to estimate 
\begin{multline}
T_{\lambda,A}u= \frac1{2} \chiev_{\lambda}(1/A)\chiev_{\lambda}(1/A)^{\star}u \\ =
2 \iint{}e^{-i\lambda{}A(t-s)} \cos \big( t{}A\sqrt{\Delta} \big) \cos \big(s{}A\sqrt{\Delta} \big) \hat{\chi}(t)\hat{\chi}(s)udtds \\
= \iint{}e^{-i\lambda{}A(t-s)} \Big( \cos \big( (t+s) {}A\sqrt{\Delta} \big) + \cos \big((t-s){}A\sqrt{\Delta} \big) \Big) \hat{\chi}(t)\hat{\chi}(s)udtds .
\end{multline}
We will separate the operator into two pieces, one capturing short time propagation and the other long time propagation. We wish to set the problem up so that the short time piece captures behaviour limited to one fundamental domain while the long time pieces captures the behaviour across multiple domains. The scaling $A$ can be seen as a time scaling so short time in this setting is $A|t-s|\leq{}2\epsilon$ where $\epsilon$ is some small parameter. Accordingly let $\zeta$ be a smooth cut off function supported in $[-2\epsilon,2\epsilon]$ and equal to one on $[-\epsilon,\epsilon]$. We decompose $T_{\lambda,A}$ into the operators $T^{1}_{\lambda,A}$, $T^{2}_{\lambda,A}$ and $T^{3}_{\lambda,A}$ given by
\begin{equation}
T^{1}_{\lambda,A}=\iint{}e^{-i\lambda{}A(t-s)}\cos \big(A\sqrt{\Delta}(t-s) \big) \hat{\chi}(t)\hat{\chi}(s)\zeta(A(t-s)) \, dtds\label{T1def}
\end{equation}
\begin{equation}
T^{2}_{\lambda,A}=\iint{}e^{-i\lambda{}A(t-s)}\cos \big(\edit{A}\sqrt{\Delta}(t-s) \big) \hat{\chi}(t)\hat{\chi}(s)(1-\zeta(A(t-s))) \, dtds 
\label{T2def}
\end{equation}
and 
\begin{equation}
T^{3}_{\lambda,A}=
\iint{}e^{-i\lambda{}A(t-s)}\cos \big(\edit{A}\sqrt{\Delta}(t+s) \big) \hat{\chi}(t)\hat{\chi}(s) dtds.
\label{T3def}
\end{equation}
We treat $T^{1}_{\lambda,A}$ first. 

\begin{prop}\label{T1estimate}
Let $T^{1}_{\lambda,A}$ be given by \eqref{T1def}. Then for any $A\geq{}1$
\begin{equation}\norm{T^{1}_{\lambda,A}u}_{L^{p}}\leq{}C\frac{\lambda^{2\delta(n,p)}}{A}\norm{u}_{L^{p'}}.
\label{T1est}\end{equation}
\end{prop}

We give a self-contained proof of Proposition~\ref{T1estimate}. A sketch of a shorter proof which uses Sogge's estimate for unit length spectral windows as an input is given in Remark~\ref{fc}. This is similar to the argument used in \cite{bourgain12}. 

\begin{proof}
We have set the cut off function $\zeta$ such that $T^{1}_{\lambda,A}$ captures the contribution to $T_{\lambda,A}$ arising from one fundamental domain; indeed, using the finite speed of propagation of the cosine kernel and the support properties of $\zeta$, we see that its kernel is supported where $d(x,y) \leq 2\epsilon$. Therefore we do not need to sum over shifted solutions on $\widetilde{M}$ (the terms $T^{2}_{\lambda,A}$ and $T^{3}_{\lambda,A}$ will require such a sum). Hence, for this term, we choose to express $\cos\big(\sqrt{\Delta}(t) \big) $ in terms of $e^{it\sqrt{\Delta}}$ and $e^{-it\sqrt{\Delta}}$. In the short time setting this use of the half wave kernel is advantageous. %particularly when $(t-s)$ is close to zero.

Applying a change of variable $t\to{}tA$ and $s\to{}sA$ we have
\begin{equation}
T^{1}_{\lambda,A}= \frac{1}{2A^{2}}\iint{}e^{-i\lambda{}(t-s)} \Big( \sum_{\pm}  e^{\pm i\sqrt{\Delta}(t-s)} \Big)  \hat{\chi}\left(\frac{t}{A}\right)\hat{\chi}\left(\frac{s}{A}\right)\zeta((t-s)) \zeta \left(\frac{d(x,y)}{2} \right) \, dt \, ds. \label{T1rescale}
\end{equation}
%\Big( e^{i\sqrt{\Delta}(t-s)} + e^{-i\sqrt{\Delta}(t-s)} \Big)
Let us write $T^{1, \pm}_{\lambda,A}$ for the expression above where we replace the sum $\sum_{\pm}$ by either the $+$ or the $-$ term. Thus 
$T^{1}_{\lambda,A} = T^{1,+}_{\lambda,A} + T^{1,-}_{\lambda,A}$. We only treat the term $T^{1,+}_{\lambda,A}$ below ($T^{1,-}_{\lambda,A}$ is trivial to estimate as the integral corresponding to \eqref{T1oscint} below has no stationary points). 

We therefore need an expression for $e^{ i(t-s)\sqrt{\Delta}}$ where $|t-s|\leq{}2\epsilon$. In this case we may use the  short time parametrix of H\"{o}rmander \cite{hormander68} (see also Sogge \cite[Section 4.1]{sogge93})
$$e^{i(t-s)\sqrt{\Delta}}u=\int{}e^{i(\phi(x,y,\xi)-(t-s)|\xi|_{g(y)})}a(t,x,y,\xi)u(y)dyd\xi$$
where in local coordinates $|\xi|_{g(x)}=\sqrt{g^{ij}(x)\xi_{i}\xi_{j}}$ and $\phi$ satisfies the Hamilton-Jacobi equation 
$$|\nabla_{x}\phi|_{g(x)}=|\xi|_{g(y)}$$
with initial condition 
$$\phi(x,y,\xi)=0\;\mbox{when}\;\langle{}x-y,\xi\rangle=0,$$
while $a$ satisfies 
$$a(0,x,x,\xi)-1\in{}S^{-1}.$$
 Rescaling the time variables $s\to{}As$, $t\to{}At$ we can write the Schwartz kernel of $T^{1,+}_{\lambda,A}$ as 
\begin{equation}
\int{}e^{-i\lambda{}(t-s)}e^{i(\phi(x,y,\xi)-(t-s)|\xi|)}a(t-s,x,y,\xi)\hat{\chi}\left(\frac{t}{A}\right)\hat{\chi}\left(\frac{s}{A}\right)\zeta(t-s) \zeta \left(\frac{d(x,y)}{2} \right)  \frac{dt \, ds \, d\xi}{A^2}
\label{T1oscint}\end{equation}
where $|\xi| = |\xi|_{g(y)}$.  
Changing variables $s'={}t-s$ and $\xi\to\lambda\xi$ we can write this in the form

\begin{equation}\begin{gathered}
T^{1,+}_{\lambda,A}=\frac{1}{A^{2}} \int_{\epsilon A}^{2\epsilon A}  \hat{\chi}\left(\frac{t}{A}\right) 
T^{1,+}_{\lambda,A}(t) \, dt, \\
T^{1,+}_{\lambda,A}(t) := \lambda^n 
\int{}e^{-i\lambda(s'+\phi(x,y,\xi)-s'|\xi|)}\tilde{a}(s',x,y,\xi)\hat{\chi}\left(\frac{t-s'}{A}\right)\zeta(s') \, ds' \, d\xi.
\end{gathered}\label{T1tdef}\end{equation}

Here and in what follows we will abuse notation somewhat and refer to both an operator and its Schwartz kernel with the same notation. To establish \eqref{T1est}, it thus suffices to obtain a bound of the form 
\begin{equation}
\big\| T^{1,+}_{\lambda,A}(t) \big\|_{L^{p'} \to L^p} \leq C {\lambda^{2\delta(n,p)}}.
\label{T1lambdaest}\end{equation}

We will do this, following the strategy in \cite{sogge93}, by applying the stationary phase lemma to the integrals in \eqref{T1tdef}. 
We first note that  if we localize
the integral in the definition of $T^{1,+}_{\lambda,A}(t)$ away from $|\xi|_{g(y)}= 1$ then we see that the phase is nonstationary in $s'$. It follows (by integrating by parts in the $s'$ variable) that the integral is $O(\lambda^{-N})$ for every $N$. Therefore we may assume that the symbol $\tilde{a}(x,\xi)$ is supported where $1 - \epsilon \leq |\xi|_{g(y)} \leq 1 + \epsilon$. We may also assume that $\tilde a$ is supported where $d(x,y) \leq 2\epsilon$. 

To prepare for stationary phase calculations, for each fixed  $y$ we may choose a coordinate system centered at $y$ such that $|\xi|_{g(y)}=|\xi|_{\R^{n}}=|\xi|$. 
To estimate the kernel of ${T}^{1,+}_{\lambda,A}(t)$ we first perform an angular integration in $\xi$. That is we decompose $\xi$ into polar coordinates $\xi=(r,\omega)$. In these coordinates
$$\phi(x,y,\xi)=\langle{}x-y,\xi\rangle+O(|\xi||x-y|^{2})$$
$$=|x-y|r\cos(\theta(\omega))+O(r|x-y|^{2})$$
for $\theta(\omega)$ the angle between $(x-y)$ and $\xi$. Therefore the critical points occur when
$$|x-y|r\sin(\theta(\omega))\theta_{\omega_{i}}=O(r|x-y|^{2})$$
and the Hessian of $\phi$ in $\omega$ has lower bound
$$ \partial^{2}_{\omega \omega}\phi >  \frac{|x-y|r}{2}   \Id.
$$ 
We now obtain by stationary phase an expression for  ${T}^{1,+}_{\lambda,A}(t)$ of the form 
\begin{multline*}
\edit{\lambda^{n}}\int{}e^{i\lambda(s'+\phi(x,y,r,\omega(r,x,y))-s'r)}\tilde{\tilde{a}}(s',x,y,r,\lambda)\hat{\chi}\left(\frac{t-s'}{A}\right)  \zeta(s)ds'dr,
\end{multline*}
where $\omega(r,x,y)$ is the critical point of $\phi$ in the $\omega$ variable for fixed $(r, x, y)$ and
$$\left|\partial_{s'}^{\alpha}\partial_{x}^{\beta}\partial_{y}^{\gamma}\partial_{r}^{\kappa}\tilde{\tilde{a}}(s'x,y,r,\lambda)\right|\leq{}C_{\alpha,\beta,\gamma,\kappa}(1+\lambda d(x,y))^{-\frac{n-1}{2}}\left(\frac{\lambda}{1+\lambda{}d(x,y)}\right)^{|\beta|+|\gamma|}$$
as follows from \cite[Section 1.1]{sogge93} for example.  
We have 
 $\phi(x,y,r,\omega(r,x,y))=|x-y|r+O(r|x-y|^{2})$, so $\partial_r \phi(x,y,r,\omega(r,x,y)) = O(d(x,y))$. It follows that there is a nondegenerate stationary point of the phase in the $(s', r)$ variables, for $s = |x-y|+O(|x-y|^{2})$, $r = 1$ . Another application of stationary phase gives 
$$\lambda^{\edit{n-1}}e^{i\lambda{}\psi(x,y)}b(x,y,t,\lambda)$$
where
$$\psi(x,y)=|x-y|+O(|x-y|^{2})$$
is the Riemannian distance between $x$ and $y$, and
$$\left|\partial_{x}^{\beta}\partial_{y}^{\gamma}b(x,y,t,\lambda)\right|\leq{}C_{\beta,\gamma}(1+\lambda d(x,y))^{-\frac{n-1}{2}}\left(\frac{\lambda}{1+\lambda{}d(x,y)}\right)^{|\beta|+|\gamma|}.$$
Such a kernel is similar in form to those operators analyzed by Sogge \cite[Section 2.2]{sogge93} by calculating $\chi_{\lambda}(1)\chi^{\star}_{\lambda}(1)$,  and therefore we can expect them to have the same  $L^{p'}\to{}L^{p}$ mapping norm of $\lambda^{2\delta(n,p)}$. This can be easily shown by following Sogge's ``freezing'' argument, found in \cite[Section 2.2]{sogge93}. This argument is often used to prove eigenfunctions estimates (see for example \cite{tataru98}, \cite{koch07}, \cite{tacy10}). We first excise those points within $\lambda^{-1}$ of the diagonal. That is
$$T^{1,+}_{\lambda,A}(t)=W_{\lambda}+R_{\lambda}, \ \text{ with }$$
$$W_{\lambda}=\lambda^{\edit{n-1}}e^{i\lambda{}\psi(x,y)}b(x,y,t,\lambda)(1-\zeta(\lambda{}d(x,y))) \ \text{ and } $$
$$R_{\lambda}=\lambda^{\edit{n-1}}e^{i\lambda{}\psi(x,y)}b(x,y,t,\lambda)\zeta(\lambda{}d(x,y)).$$
By interpolating between $L^2 \to L^2$ and $L^1 \to L^\infty$, we see that $R_{\lambda}$ has $L^{p'}\to{}L^{p}$ mapping norm of at most  $\lambda^{(n-1)-\frac{2n}{p}}=\lambda^{2\delta(n,p)}$ and therefore we may restrict our attention to $W_{\lambda}$. Further we place cut off functions to define a principal direction. That is for any given direction $\eta$ we pick a coordinate system such that $\eta$ corresponds with the $x_{1}$ axis. Then writing $x=(x_{1},x')$ we have
$$\widetilde{W}_{\lambda} = \lambda^{\edit{n-1}}e^{i\lambda{}\psi(x,y)}b(x,y,t,\lambda)(1-\zeta(\lambda{}d(x,y)))\zeta\left(\frac{|x'-y'|}{d(x,y)}\right).$$
Note that on the support of the integrand $|x_{1}-y_{1}|\sim{}d(x,y)$. Since we may express $W_{\lambda}$ as a sum of a finite number of such operators we need only estimate $\widetilde{W}_{\lambda}$. We freeze the variables $x_{1}$ and $y_{1}$ and consider the operator
\begin{multline}\widetilde{W}_{\lambda}(x_{1},y_{1})=\lambda^{\frac{n-1}{2}}e^{i\lambda{}\psi(x_{1},x',y_{1},y')}b(x_{1},y_{1},x',y',t,\lambda)\\ 
\times\big(1-\zeta(\lambda{}|x_1 - y_1|)\big)\zeta\left(\frac{|x'-y'|}{|x_{1}-y_{1}|}\right).\label{Wtildedef}\end{multline}
We may read off a $L^{1}\to{}L^{\infty}$ estimate from the pointwise kernel bounds:
\begin{equation}
\norm{\widetilde{W}_{\lambda}(x_{1},y_{1})u}_{L^{\infty}}\lesssim{}\lambda^{\frac{n-1}{2}}(\lambda^{-1}+|x_{1}-y_{1}|)^{-\frac{n-1}{2}}\norm{u(y_{1},\cdot)}_{L_{y'}^{1}}.
\label{kernelbounds}\end{equation}
Following Sogge \cite[Section 2]{sogge93} we may obtain the required $L^{p'}\to{}L^{p}$ estimates via interpolation and Hardy-Littlewood-Sobolev if we also have
$$\norm{\widetilde{W}_{\lambda}(x_{1},y_{1})u}_{L^{2}_{x'}}\lesssim{}\norm{u(y_{1},\cdot)}_{L_{y'}^{2}}.$$

Thus the proof of Proposition~\ref{T1estimate} is completed by the following lemma. 
\end{proof}

\begin{lem}\label{L2lem}
Suppose $|x_{1}-y_{1}|>K\lambda^{-1}$ and $\widetilde{W}_{\lambda}(x_{1},y_{1})$ be given by \eqref{Wtildedef}. 
Then
$$\norm{\widetilde{W}_{\lambda}(x_{1},y_{1})u}_{L_{x'}^{2}}\leq{}C\norm{u(y_{1},\cdot)}_{L^{2}_{y'}}$$
where the constant $C$ is uniform in $x_{1}$, $y_{1}$. 
\end{lem}

\begin{proof}
We have
\begin{equation}\norm{\widetilde{W}_{\lambda}(x_{1},y_{1})u}_{L_{x'}^{2}}^{2}=\int{}\widetilde{K}(t,x_{1},y_{1},y',z')u(y_{1},y')\bar{u}(z_{1},z')dy'dz'\label{Wsquare}\end{equation}
where
\begin{multline}\widetilde{K}(t,x_{1},y_{1},y',z')=\lambda^{n-1}|x_{1}-y_{1}|^{-(n-1)}\int{}e^{i\lambda(\psi(x_{1},x',y_{1},y')-\psi(x_{1},x',y_{1},z')}\\ 
\times\tilde{b}(t,x_{1},y_{1},x',y',z',\lambda)\zeta\left(\frac{|x'-y'|}{|x_{1}-y_{1}|}\right)\zeta\left(\frac{|x'-z'|}{|x_{1}-y_{1}|}\right)dx'\label{Kdef}\end{multline}
where
$$\left|\partial_{x'}^{\alpha}b(t,x_{1},y_{1},x',y',z',\lambda)\right|\leq{}C_{\alpha}\left(\frac{\lambda}{1+\lambda{}d(x,y)}\right)^{|\alpha|}\lesssim{}\frac{1}{|x_{1}-y_{1}|^{|\alpha|}}.$$
We seek to estimate $|\widetilde{K}(t,x_{1},y_{1},y',z')|$ via integration by parts in $x'$. Therefore we expand the phase function
$$\psi(x_{1},x',y_{1},y')-\psi(x_{1},x',y_{1},z')$$
as a Taylor series about $z'=y'$. That is
$$\psi(x_{1},x',y_{1},y')-\psi(x_{1},x',y_{1},z')=\nabla_{y'}\psi(x_{1},x',y_{1},y')\cdot{}(y'-z')+O(|y'-z'|^{2}).$$ 
Now differentiating in $x'$ we have that
$$\nabla_{x'}[\psi(x_{1},x',y_{1},y')-\psi(x_{1},x',y_{1},z')]=[\partial^{2}_{x'y'}\psi](y'-z')+O(|y-z'|^{2})$$
As
$$\psi(x_{1},x',y_{1},y')=|x-y|+O(|x-y|^{2})$$ and 
$|x'-y'|\leq{}\epsilon{}|x-y|$, we have  
$$\partial^{2}_{x'y'}=\frac{-1}{|x-y|}\left[\Id+O(\epsilon)\right]$$
and therefore 
$$\left|\nabla_{x'}[\psi(x_{1},x',y_{1},y')-\psi(x_{1},x',y_{1},z')]\right|>c\frac{|y'-z'|}{|x_{1}-y_{1}|}.$$
Therefore integrating \eqref{Kdef} by parts we gain a factor of $c\lambda^{-1}|y'-z'|^{-1}|x_{1}-y_{1}|$ for each iteration and at worst differentiating produces growth of $|x_{1}-y_{1}|^{-1}$. Therefore we have 
\begin{multline}|\widetilde{K}(t,x_{1},y_{1},y',z')|\leq{}C_{N}\lambda^{n-1}|x_{1}-y_{1}|^{-(n-1)}\left(1+\lambda|y'-z'|\right)^{-N}\\
\times\int{}\tilde{b}(t,x_{1},y_{1},y',z',\lambda)\zeta\left(\frac{|x'-z'|}{|x_{1}-y_{1}|}\right)dx'\\
\leq{}C_{N}\lambda^{\edit{n-1}}(1+\lambda|y'-z'|)^{-N}\label{Kdef2}\end{multline}
as the integrand in \eqref{Kdef2} is only supported on a region $|x'-y'|\leq{}\epsilon|x_{1}-y_{1}|$. Therefore we apply H\"{o}lder and Young to \eqref{Wsquare} to obtain
$$\norm{\widetilde{W}_{\lambda}(x_{1},y_{1})u}_{L_{x'}^{2}}^{2}\lesssim{}\norm{u(y_{1},y')}_{L_{y'}}^{2}$$

as required. 

\end{proof}

\begin{remark}\label{T1remark}

For this small $t-s$ term there is no restriction on how large the parameter $A$ can be. It is from the large $t-s$ term that will will arrive at the restriction that $A$ must grow logarithmically in $\lambda$.
\end{remark}

\begin{remark}\label{fc} One can give an alternative proof of Proposition~\ref{T1estimate}  by showing that $T^{1,+}_{\lambda, A}(t)$ is a Schwartz function $\phi_A(\sqrt{\Delta} - \lambda)$ of $\sqrt{\Delta} - \lambda$, where $\phi_A$ is uniformly bounded in $A \geq 1$ in the sense that every seminorm is uniformly bounded for $A \geq 1$. Then one can show that any Schwartz function $\phi$ of $\sqrt{\Delta} - \lambda$ is bounded from $L^{p'}$ to $L^p$ with norm estimate $C\lambda^{2\delta(n,p)}$, where $C$ depends only on a finite number of seminorms of $\phi$. This is essentially equivalent to the argument in Lemma 2.3 of \cite{bourgain12}. 
\end{remark}

\begin{remark} One can avoid the use of Lemma~\ref{L2lem} with the observation that 
\eqref{T1tdef} is a semiclassical Fourier Integral operator. In fact, if we localize the symbol in \eqref{T1tdef} in the variable $\omega = \xi/|\xi|$ near some direction, which (by rotating in the $\xi$ coordinate) we may assume to be the $e_1$ direction, and then fix the variables $x_1$, $y_1$, the resulting frozen operator, acting in the remaining $n-1$ variables, is an FIO of order zero associated to a canonical graph, and therefore $L^2$ bounded. This localized and frozen operator is analogous to $\widetilde{W_\lambda}$ above, with the only difference being that the localization is done in the $\xi$ variables before stationary phase, instead of in the spatial variables directly. 
\end{remark}

We now treat $T^{2}_{\lambda,A}$.

\begin{prop}\label{T2estimate}
Suppose $T^{2}_{\lambda,A}$ is given by \eqref{T2def} then for any $p$ there exists an $\alpha_{p}$ and a $C_{p}$ such that
with $A = \alpha_p \log \lambda$ we have 
$$\norm{T^{2}_{\lambda,A}u}_{L^{p}}\leq{}C_{p}\frac{\lambda^{2\delta(n,p)}}{\log\lambda}$$
\end{prop}

\begin{proof}
We have
$$T^{2}_{\lambda,A}u=\iint{}e^{-i\lambda{}A(t-s)}\cos \big(\edit{A}(t-s)\sqrt{\Delta} \big)\hat{\chi}(t)\hat{\chi}(s)(1-\zeta(A(t-s)))udtds$$
After scaling $t\to{}At$ and $s\to{}As$ and setting $s' = t-s$  we obtain
$$T^{2}_{\lambda,A}u=\frac{1}{A^{2}}\iint{}e^{-i\lambda{}s'}\cos \big(s'\sqrt{\Delta} \big) \hat{\chi}\left(\frac{t}{A}\right)\hat{\chi}\left(\frac{t-s'}{A}\right)(1-\zeta(s'))u \, dt \, ds'.$$

We estimate this term by interpolating between $L^2 \to L^2$ estimates and $L^1 \to L^\infty$ estimates. Using the bound $|\cos t| \leq 1$ we get a trivial  $L^2 \to L^2$ estimate for $T^{2}_{\lambda,A}$: 
\begin{equation}
\big\| T^{2}_{\lambda,A} \big\|_{L^2 \to L^2} \leq C,
\label{L2L2}\end{equation}
uniformly in $\lambda$ and $A$. 

An $L^1 \to L^\infty$ estimate on $T^{2}_{\lambda,A}$ is equivalent to a pointwise kernel estimate. To obtain such an estimate we express the cosine kernel on $M$ in terms of the cosine kernel on 
the universal cover $\widetilde{M}$, following B\'{e}rard. This exploits the  fact that the universal cover $\widetilde{M}$ has no conjugate points. B\'{e}rard \cite[Section D]{berard77} expresses the cosine kernel $\tilde{e}(t,x,y)$ on $\widetilde{M}$ in terms of a finite series plus an error term:
\begin{equation}
\tilde{e}(t,x,y) = C_0 \sum_{k=0}^N (-1)^k 4^{-k}  u_k(x,y) |t| \;\chi_+^{k- (n+1)/2}\Big(d(x,y)^2 - t^2\Big) + \tilde \epsilon_N(t, x, y),
\label{parametrix}\end{equation}
where $d(x,y)$ is the Riemannian distance between $x$ and $y$ on $\widetilde{M}$, and $\chi_+^a(r)$ is the analytic continuation of $r_+^a/\Gamma(a+1)$ in the parameter $a$ from the region $\Re a > -1$. It is shown in \cite{berard77} that under the assumptions of Theorem~\ref{maintheorem}, $u_k$ satisfies the estimate 
\begin{equation}
\Big| u_k(x, y) \Big| \leq C_{k} e^{c_{k} d(x,y)}, 
\label{ukest}\end{equation}
and that provided that $N > d+1$, $\tilde\epsilon_N$ is a continuous function  satisfying the estimate 
\begin{equation}
\Big| \tilde \epsilon_N(t, x,y) \Big| \leq C_N e^{c_N t}.
\label{epsilonNest}\end{equation}

The cosine kernel $e(t, x, y)$  on $M$ is related to that  on $\widetilde{M}$ by 
\begin{equation}
e(t, x, y) = \sum_{\gamma \in \Gamma} \tilde e(t, x, \gamma y)
\label{eetilde}\end{equation}
where $\Gamma$ is the fundamental group of $M$, acting on $\widetilde{M}$ by deck transformations. Here we make crucial use of the fact that the cosine kernel is supported where $d(x,y) \leq t$: this means that for fixed $x, y \in M$ the sum in \eqref{eetilde} is finite for each $t$. In fact, since $x, y$ may be viewed as points in $\widetilde M$ in the same fundamental domain, i.e. distance $O(1)$ apart in $\widetilde M$, the number of terms in the sum \eqref{eetilde} is $O(e^{C|t|})$ for some $C$, since $\widetilde{M}$ has bounded sectional curvatures and therefore the volume growth of balls is at most exponential. 

\begin{lem}\label{kernel-estimate} Let $\eta > 0$ be given.  Suppose that $A = \alpha \log \lambda$. Then if $\alpha$ is sufficiently small and $\lambda$ sufficiently large, we have 
\begin{equation}
\frac1{A} \bigg| \int e^{\edit{-i\lambda s'}}  e(s', x, y) \, \hat{\chi}\left(\frac{t-s'}{A}\right)(1-\zeta(s')) \, ds' \bigg| \leq C  \lambda^{(n-1)/2 + \eta} 
\label{int-est}\end{equation}
uniformly for $t \in [0, 2\epsilon A]$ and $(x, y) \in M^2$. 
\end{lem}

\begin{proof}[Proof of Lemma~\ref{kernel-estimate}]
We express the cosine kernel $e(t, x, y)$ via \eqref{eetilde} and decompose $\tilde e$ as in \eqref{parametrix}, where we choose $N = n/2 +1$ if $n$ is even and $N = n/2 + 1/2$ if $N$ is odd. Let
$$
\epsilon_N(s', x, y) = \sum_{\gamma \in \Gamma} \tilde \epsilon_N(s', x, \gamma y).
$$
Then since $\hat \chi$ is supported in $[\epsilon, 2\epsilon]$, the
integrand in \eqref{int-est} is zero unless $|s'| \leq 2A \epsilon$.
In that case, due to the estimate \eqref{epsilonNest} and the exponential estimate of the number of nonzero terms in this sum, we find that 
$$
|\epsilon_N(s', x, y)| \leq C e^{(C_N + C') \alpha A},
$$
showing that for $A = \alpha \log \lambda$ and $\alpha$ sufficiently small, the contribution of the $\epsilon_N$ term to \eqref{int-est} is $O(\lambda^{\eta})$ for any $\eta > 0$. 

It therefore suffices to check that the contribution of each of the $u_k$ terms in \eqref{parametrix} gives a $O(\lambda^{(n-1)/2 + \eta})$ contribution to \eqref{int-est}. 

For each $k \leq N$, we consider the term 
\begin{equation}
e_k(t, x, y) = \sum_{\gamma \in \Gamma} \tilde e_k(t, x, \gamma y)
\label{eetildek}\end{equation}
where $\tilde e_k$ is the $k$th term in the sum \eqref{parametrix}. Since, as we have already observed, there are at most $e^{c\alpha \log \lambda}$ terms in this sum, and $e^{c\alpha \log \lambda} \leq C \lambda^{\eta/3}$ for sufficiently small $\alpha$, it suffices to estimate a single term in this sum over $\gamma \in \Gamma$. Similarly, we have an estimate $$|u_k( x, y)| \leq C_k e^{c_{k}  d(x,y)} \leq C_k e^{c_{k}  s' } \leq C_k e^{c_k \epsilon \alpha \log \lambda}$$ on the support of this term, again using finite propagation speed,  which is bounded by  $C \lambda^{\eta/3}$
 for sufficiently small $\alpha$. So it suffices to obtain an estimate 
\begin{equation}
\frac1{A} \bigg| \int e^{\edit{-i\lambda s'}} \chi_+^{k- (d+1)/2}(d(x, \gamma y)^2 - {s'}^2) \, \hat{\chi}\left(\frac{t-s'}{A}\right)(1-\zeta(s')) \, ds' \bigg| \leq C  \lambda^{(n-1)/2 + \eta/3 } .
\label{int-est-k}\end{equation}
To verify this for odd $n$ we split the integral into the parts where $s' > 0$ and where $s' < 0$ and consider first the former. We
use the fact that for positive integers $m$, the distribution $\chi_+^{-m}$ is equal to $\delta^{(m-1)}$, the $(m-1)$th derivative of the delta function, which is homogeneous of degree $-m$, and hence 
$$
\chi_+^{k- (n+1)/2}(d(x, \gamma y)^2 - {s'}^2) = \delta^{(n-1)/2-k} (d(x, \gamma y) - {s'}) (d(x, \gamma y) + {s'})^{k- (n+1)/2}. 
$$
The integral therefore involves taking $(n-1)/2-k$ derivatives in $s'$ of 
$$
e^{\edit{-i\lambda s'}} (d(x, \gamma y) + {s'})^{k- (n+1)/2} \, \hat{\chi}\left(\frac{t-s'}{A}\right)(1-\zeta(s'))
$$
and then setting $s' = d(x, \gamma y)$. This is clearly bounded by $\lambda^{(n-1)/2}$ times a power of $d(x, \gamma y)$ and since we have
$d(x, \gamma y) \leq 2\epsilon \alpha \log \lambda$ by finite propagation speed, we see that this is bounded by $C\lambda^{(n-1)/2 + \eta/3 }$. The integral for $s' < 0$ is estimated in exactly the same way. 

For even $n$, we write 
\begin{multline}
\chi_+^{k- (n+1)/2}(d(x, \gamma y)^2 - {s'}^2) = c_{k,n} \chi_+^{k- (n+1)/2}(d(x, \gamma y) - {s'}) (d(x, \gamma y) + {s'})^{k- (n+1)/2} \\
=  \bigg( \Big( \frac{d}{ds'} \Big)^{n/2 - 1 - k} \chi_+^{-3/2}(d(x, \gamma y) - s') \bigg) (d(x, \gamma y) + {s'})^{k- (n+1)/2}.
\end{multline}
We use this in the integral \eqref{int-est-k} and integrate by parts  $n/2 - 1 - k$ times. The integrand becomes  a sum of terms each of which takes the form 
\begin{equation}
\frac{c_{k,n}}{A^{-j_3}} \lambda^{j_1} e^{\edit{-i\lambda s'}} \chi_+^{-3/2}(d(x, \gamma y) - s') 
(d + s')^{-(n+1)/2 +k - j_2} \hat{\chi}^{(j_3)}\left(\frac{t-s'}{A}\right) (1 - \zeta)^{(j_4)}(s')
\label{integrand-terms}\end{equation}
where $j_1, j_2, j_3, j_4 \geq 0$ and $j_1 + j_2 + j_3 +j_4 = n/2 - 1- k$.
Let us write this term 
$$
A^{-j_3}\lambda^{j_1} e^{\edit{-i\lambda s'}} \chi_+^{-3/2}(d(x, \gamma y) - s') 
B(s'),
$$
i.e. $B(s') = (d + s')^{-(n+1)/2 +k - j_2} \hat{\chi}^{(j_3)}\left(\frac{t-s'}{A}\right) (1 - \zeta)^{(j_4)}(s')$. 
We isolate the singularity of $\chi_+^{-3/2}(d - s')$, $d = d(x, \gamma y)$,  by writing 
\begin{equation}\begin{gathered}
\chi_+^{-3/2}(d - s') B(s')= \\
 \chi_+^{-3/2}(d - s') \zeta(d - s') B(d)  + 
\chi_+^{-3/2}(d - s') \Big( B(s') - \zeta(d - s') B(d) \Big) .
\end{gathered}\label{sing}\end{equation}
We substitute \eqref{sing} into \eqref{integrand-terms} and integrate in $\lambda$. Taking the first term on the right hand side of \eqref{sing}, the integral is equal to  
$$
e^{i\lambda d} \Big( (\lambda - i0)^{j_1 + 1/2}  * \hat \zeta(\lambda) \Big) B(d),
$$
and is bounded by for large $\lambda$ by $\lambda^{(n-1)/2}$ times a power of $d(x, y)$, hence bounded by $\lambda^{(n-1)/2 + \eta/3}$ when $\alpha$ is sufficiently small. Taking now the second term on the right hand side of \eqref{sing}, observe that the term $B(s') - \zeta(d - s') B(d)$ is smooth and vanishes at $s' = d$, hence the singularity at $s' = d$ is reduced to an integrable singularity. Therefore the integrand in this case is a polynomial in $s'$ times a function that is absolutely integrable uniformly in the parameters $x, y, t$. Now for $t\in[0,2\epsilon{}A]$, $\hat{\chi}\left(\frac{t-s}{A}\right)$ is supported only for $|s'|\leq{}4\epsilon{}A$ so this term contributes $\lambda^{n/2 - 1 + \eta/3}$, which is smaller than required. This completes the proof of Lemma~\ref{kernel-estimate}. 

\end{proof}

\noindent\emph{Continuation of proof of Proposition~\ref{T2estimate}.} 
It is an immediate consequence of Proposition~\ref{kernel-estimate} that we have a kernel bound 
\begin{equation}
\big\| T^{2}_{\lambda,A} \big\|_{L^1 \to L^\infty} \leq C \lambda^{(n-1)/2 + \eta},
\label{L1Linf}\end{equation}
for any $\eta > 0$. 
Interpolating between \eqref{L2L2} and \eqref{L1Linf} we obtain
$$\norm{T^{2}_{\lambda,A}u}_{L^{p}}\lesssim \lambda^{\frac{n-1}{2}-\frac{n-1}{p} + \eta(1-\frac{2}{p})}\norm{u}_{L^{p'}}, \quad p > 2$$
for any $\eta > 0$, which we can write using \eqref{deltanp}
$$\norm{T^{2}_{\lambda,A}u}_{L^{p}}\lesssim{\lambda^{2\delta(n,p)-\frac{n-1}{2}+\frac{n+1}{p} + \eta(1-\frac{2}{p})}}\norm{u}_{L^{p'}}. $$
Therefore, provided we have 
$$-\frac{n-1}{2}+\frac{n+1}{p}<0$$
or $p>\frac{2(n+1)}{n-1}$
we can choose $\eta$ sufficiently small so that $-\frac{n-1}{2}+\frac{n+1}{p} + \eta(1-\frac{2}{p}) < 0$, and we find that there exists $\eta' = \eta'(p) > 0$ such that 
$$\norm{T^{2}_{\lambda,A}u}_{L^{p}}\lesssim \lambda^{2\delta(n,p)-\eta'}\norm{u}_{L^{p'}}$$
which is a stronger result than claimed in Proposition~\ref{T2estimate}. 
\end{proof}

We finally have
\begin{prop}\label{T3estimate}
Suppose $T^{3}_{\lambda,A}$ is given by \eqref{T3def} then for any $p$ there exists an $\alpha_{p}$ and a $C_{p}$ such that with $A = \alpha_p \log \lambda$ we have 
$$\norm{T^{3}_{\lambda,A}u}_{L^{p}}\leq{}C_{p}\frac{\lambda^{2\delta(n,p)}}{\log\lambda}$$
\end{prop}

This proposition is proved in exactly the same way as for Proposition~\ref{T2estimate}, so we omit the details. (Note that in \eqref{T3def}, we have $t, s \in [\epsilon, 2\epsilon]$ on the support of the integrand, so that $t + s \geq 2\epsilon$; in particular, $t+s$ is bounded away from zero so we only need consider the cosine kernel for times bounded away from zero, just as for $T^{2}_{\lambda,A}$.)

Putting together Propositions~\ref{T1estimate}, \ref{T2estimate} and \ref{T3estimate} with the choice $A = \alpha_p \log \lambda$, we obtain the proof of Theorem~\ref{maintheorem}.

\bibliography{logpaper_ref}
\bibliographystyle{abbrv}

\end{document}